\documentclass[a4paper]{amsart}

\usepackage{amsmath, amsthm, amsfonts, amssymb}
\usepackage{colonequals}
\usepackage{paralist}
\usepackage{tikz}

\newcommand{\s}{\mbox{\rm s}}
\newcommand{\cq}{\colonequals}
\newcommand{\hansen}{\mbox{\rm H}}
\newcommand{\tprism}{\mbox{\rm tp}}
\newcommand{\stab}{\mbox{\rm stab}}
\newcommand{\conv}{\mbox{\rm conv}}
\newcommand{\vertex}{\mbox{\rm vert}}

\newcommand{\rr}{\mathbb{R}}

\newcommand{\dotcup}{\ensuremath{\mathaccent\cdot\cup}}
\newcommand{\disjoint}{\,\dotcup\,}
\newcommand{\ve}{\varepsilon}

\newtheorem{theorem}{Theorem}[section]
\newtheorem{lemma}[theorem]{Lemma}
\newtheorem{proposition}[theorem]{Proposition}
\newtheorem{corollary}[theorem]{Corollary}
\newtheorem*{conjecture*}{Conjecture}
\theoremstyle{definition}
\newtheorem{definition}[theorem]{Definition}
\theoremstyle{remark}

\title[Face numbers of centrally symmetric polytopes from split graphs]{Face numbers of centrally symmetric polytopes produced from split graphs}

\author[R.\ Freij]{Ragnar Freij}
\address{Chalmers University of Technology and University of Gothenburg\\ S-412 96 G\"o\-te\-borg, Sweden}
\email{freij@chalmers.se}

\author[M.\ Henze]{Matthias Henze}
\address{Fakult\"at f\"ur Mathematik, Otto-von-Guericke-Universit\"at Magdeburg\\ Uni\-ver\-si\-t\"ats\-platz 2\\ 39106 Magdeburg, Germany}
\email{matthias.henze@ovgu.de}

\author[M.\ W.\ Schmitt]{Moritz W.\ Schmitt}
\address{Institut f\"ur Mathematik\\ Freie Universit\"at Berlin\\ Arnimallee 2\\ 14195 Berlin, Germany}
\email{mws@math.fu-berlin.de}

\author[G.\ M.\ Ziegler]{G\"unter M.\ Ziegler}
\address{Institut f\"ur Mathematik\\ Freie Universit\"at Berlin\\ Arnimallee 2\\ 14195 Berlin, Germany}
\email{ziegler@math.fu-berlin.de}

\thanks{The work of MH was supported by the Deutsche Forschungsgemeinschaft (DFG) within the project He 2272/4-1.
The work by MWS and GMZ was supported by the
DFG Research Center \textsc{Matheon}, Project F13, and by ERC Advanced Grant agreement no.~247029-SDModels.}

\date{}

\begin{document}

\begin{abstract}
We analyze a remarkable class of centrally symmetric polytopes,
 the Hansen polytopes of split graphs. We confirm
Kalai's $3^d$-conjecture for such polytopes (they all have at least $3^d$ nonempty faces) 
and show that the Hanner polytopes among them (which have exactly $3^d$ nonempty faces) 
correspond to threshold graphs. Our study produces a new
family of  Hansen polytopes that have only $3^d+{16}$ nonempty faces.
\end{abstract}


\maketitle

\section{Introduction}

A convex polytope $P$ is \emph{centrally symmetric} if $P=-P$. 
In 1989 Gil Kalai \cite{kalai1989number}  posed three conjectures on the
numbers of faces and flags of centrally symmetric polytopes, which he named
conjectures A, B and C. Two of these, conjectures B and C, were refuted by
Sanyal et al.\ in 2009 \cite{sanyal2009kalai}. However, conjecture A,
known as the $3^d$-conjecture, was confirmed for dimension $d \le4$ and remains open for $d>4$:

\begin{conjecture*}[$3^d$-conjecture]
Every centrally symmetric convex polytope of dimension $d$ has at least $3^d$
nonempty faces.
\end{conjecture*}

As a contribution to the quest for settling this
conjecture, we investigate a special class of centrally symmetric polytopes,
namely Hansen polytopes, as introduced by Hansen in 1977 \cite{hansen1977certain}.
Hansen polytopes of split graphs served as
counter-examples to conjectures B and C, so it seems natural to
analyze this subclass more thoroughly. As our main result we express the
total number of nonempty faces of such a polytope in terms of  
certain partitions of the node set of the underlying split graph. In
particular, we confirm the $3^d$-conjecture for Hansen polytopes of split
graphs, and show that equality in this class corresponds to threshold graphs.

In Section 2 we
define Hansen polytopes, which are derived from perfect graphs such as, for example, split graphs. 
In Section 3 we analyze the Hansen polytopes of threshold graphs, which are special split graphs. 
It turns out that a Hansen polytope is a Hanner polytope if and only if the underlying
graph is threshold. In Section 4 we describe the Hansen polytopes of general split graphs
and prove the main result mentioned above. Our study also produces
examples of centrally symmetric polytopes that are not
Hanner polytopes and have a total number of nonempty faces very close to the
conjectured lower bound of $3^d$.

\smallskip

\emph{General assumptions.} All our graphs are finite and simple. The
vertex set of a graph $G$ is denoted by $V(G)$, and similarly
the edge set is $E(G)$ if no other notation is specified. The complement of $G$ 
is~$\overline G$. The complete graph on $n$ nodes is~$K_n$.
All polytopes are convex. We denote the polar of a polytope $P$ by $P^\ast$. 
For details on graph theory we refer to Diestel \cite{diestel2010graph}, 
for polytope theory to~\cite{ziegler1995lectures}.

\section{Hansen Polytopes}

Hansen polytopes were introduced by Hansen \cite{hansen1977certain} in 1977.
Some of these centrally symmetric polytopes turn out to have
``few faces''. One constructs them from the stable set
structure of a (perfect) graph $G$ by applying the twisted prism operation to the stable
set polytope. Let us define these terms.

\begin{definition}[Twisted prism]
Let $P \subseteq \rr^d$ be a polytope and $Q \cq \{1\} \times P$ its
embedding in $\rr^{d+1}$. The \emph{twisted prism} of $P$ is
$\tprism(P) \cq \conv(Q \cup -Q)$.
\end{definition}

Twisted prisms are centrally symmetric by construction. We are
interested in twisted prisms of stable set polytopes, which we introduce next;
see also Schrijver \cite[Sec.\ 64.4]{schrijver2003combinatorial}. By $e_i$ we denote 
the $i$th coordinate unit vector.

\begin{definition}
Let $G$ be a graph. The \emph{stable set polytope} of $G$ is 
\[\stab(G) \cq \conv\Big\{\sum_{i \in I} e_i : I \subseteq V(G) \mbox{ stable} \Big\}.\]
\end{definition}

Now we can define the main object of our studies.

\begin{definition}[Hansen polytope]
The \emph{Hansen polytope} of a graph $G$
is defined as $\hansen(G) \cq \tprism(\stab(G))$. 
\end{definition}

Examples of Hansen polytopes are cubes (produced from empty graphs) and crosspolytopes (from complete graphs).
Recall that a graph $G$ is \emph{perfect} if the size of the largest clique of any induced subgraph $H$ of $G$ equals the chromatic number of $H$.
In the rest of the paper we need the following properties of Hansen polytopes:

\begin{lemma}\label{lemma_hansen_prop}
Let $G = (V,E)$ be a graph. 
\begin{enumerate}[\rm(i)]
\item The vertex set of $\hansen(G)$ is $\vertex(\hansen(G)) = \{ \pm(e_0
+ \sum_{i \in I} e_i) : I \subseteq V \mbox{ stable} \}$.
\item If $G$ is perfect, $\{ -1 \leq -x_0 + 2 \sum_{i \in C} x_i \leq 1 :
C \subseteq V \mbox{ clique} \}$ is an irredundant facet description of
$\hansen(G)$.
\item If $G$ is perfect, then the polar of the Hansen polytope of $G$ is affinely equivalent to the Hansen polytope of $\overline G$, in symbols $\hansen(G)^\ast \cong \hansen(\overline{G})$.
\end{enumerate}
\end{lemma}

\begin{proof}
Part (i) is obvious, a proof of (ii) can be found in Hansen's paper
\cite{hansen1977certain} and (iii) follows from (ii).
\end{proof}

From now on for the rest of the article we assume all graphs to be perfect.

\section{Hansen Polytopes of Threshold Graphs}

An important class of polytopes that attain the conjectured lower bound of the
$3^d$-conjecture are the so-called Hanner polytopes. These polytopes were
introduced by Hanner \cite{hanner1956intersections} in 1956 and are recursively
defined as follows.

\begin{definition}[Hanner polytope]
A polytope $P \subseteq \rr^d$ is a \emph{Hanner polytope} if it is
either a centrally symmetric line segment or, for $d \geq 2$, the direct
product of two Hanner polytopes or the polar of a Hanner polytope.
\end{definition}

It is neither the case that all Hanner polytopes are Hansen polytopes nor
vice versa. A characterization of their relation is our
first result. Before we can state it, we need to introduce threshold
graphs, a subclass of perfect graphs; an extensive treatment is Mahadev \&
Peled \cite{mahadev1995threshold}. The definition involves
the notions of dominating and isolated nodes: A node in a graph is
\emph{dominating} if it is adjacent to all other nodes; it is
\emph{isolated} if it is not adjacent to any other node.

\begin{definition}[Threshold graph]
A graph $G = (V,E)$ is a \emph{threshold graph} if it can be constructed
from the empty graph by repeatedly adding either an isolated node or a dominating node.
\end{definition}

This class of graphs is closed under taking complements. 

\begin{theorem}\label{hansen_threshold} 
The Hansen polytope $\hansen(G)$ is a Hanner polytope if and only if $G$ is a threshold graph.
\end{theorem}

\begin{proof}
($\Leftarrow$) 
We use induction of the number of nodes. If $G =\emptyset$, 
then $\hansen(G)$ is just a centrally symmetric segment, and
therefore a Hanner polytope. Now assume that $G$ has $n+1$ nodes. Since the class
of Hanner polytopes is closed under taking polars and $\hansen(G)^\ast \cong
\hansen(\overline{G})$, we can assume $G = T \disjoint \{v\}$ with $T$ being
threshold. Here $\!\disjoint\!$ denotes the usual disjoint union of graphs and
$v$ is a single node with $v \not\in T$. The stable sets of $G$ are exactly the
stable sets of $T$, with and without the new node $v$. Given a stable set
$S$ of $T$ the vertices of $\hansen(G)$ are of the form $\pm (e_0 + \sum_{i \in
S} e_i)$ and $\pm (e_0 + \sum_{i \in S} e_i + e_{n+1})$, where we assign $v$
the label $n+1$. By the linear transformation
defined by $e_0 \mapsto e_0 - e_{n+1}$, $e_{n+1} \mapsto 2 e_{n+1}$, and $e_i
\mapsto e_i$ for $i = 1, \ldots, n$, we get $\hansen(G) = \hansen(T \disjoint
v) \cong \hansen(G) \times [-1,1]$, which means that $\hansen(G)$ is Hanner.

($\Rightarrow$) Assume that $\hansen(G)$ is Hanner. Again it is enough to cover
just one case, namely $\hansen(G) = P \times P'$ with $P,P'$ being
lower-dimensional Hanner polytopes. The stable set polytope $\stab(G)$ is a facet
of $\hansen(G)$ and can therefore be written as $\stab(G) = Q \times Q'$ with
$Q, Q'$ being faces of $P, P'$, respectively. Since we have $\dim(Q) + \dim(Q')
= \dim(\stab(G)) = \dim(P) + \dim(P') - 1$, we can further assume
that $Q = P$ and $Q'$ is a facet of $P'$. Let $q \cq \dim(Q)$ and $q' \cq
\dim(Q')$. 

We now construct a threshold graph $H'$ on $q'$ nodes such that 
$G = \overline{K_q} \cup H'$. This of course shows that $G$ is threshold as
well. Since $\stab(G)$ is a product, we have $\vertex(\stab(G)) = \vertex(Q)
\times \vertex(Q')$. Each coordinate of a vertex of $\stab(G)$ corresponds to a
node in $G$. Let $V_1 \subseteq V(G)$ be the node set defined by the first $q$
coordinates and $V_2 \subseteq V(G)$ the set defined by the last $q'$
coordinates. Then  
\begin{align*}
\vertex(\stab(G)) &= \Big\{ \sum_{i \in I} e_i : I \subseteq V(G) \mbox{ stable set of } G \Big\}\\
 &= \Big\{ \sum_{i \in I} e_i : I \subseteq V_1 \mbox{ stable set of } G[V_1] \mbox{ and } N(I) \cap V_2 = \emptyset \Big\}\\
 &\quad \times \Big\{ \sum_{i \in I} e_i : I \subseteq V_2 \mbox{ stable set of } G[V_2] \mbox{ and } N(I) \cap V_1 = \emptyset \Big\},
\end{align*}
where $N(I)$ is the set of nodes adjacent to some node in $I$ and $G[V_j]$ is the subgraph of $G$ induced by $V_j,j=1,2$. In particular,
we have $e_i \in \vertex(\stab(G))$ for all $i=1,\dots,q+q'$. From this and the right-hand side of the
equality above, we can deduce that there are no edges between $V_1$ and
$V_2$. By setting $H' \cq G[V_2]$ we get $G = G[V_1] \cup H'$. So what is
left to show is that $G[V_1]$ is an edgeless graph and that $H'$ is threshold.

Let us first see that $H'$ is threshold. Since $P \cong \stab(G[V_1])$ is
at least one-dimensional, $G[V_1]$ has one node minimum, i.e., $|V(H')| <
|V(G)|$. From \cite[Corollary 3.4 and Theorem 7.4]{hanner1956intersections}, we know
that Hanner polytopes are twisted prisms over any of their facets, which means
for us that $P' \cong \tprism(Q') \cong \hansen(H')$.
Thus, by induction, $H'$ is threshold.

As $P \cong \stab(G[V_1])$ is Hanner, it has a center of
symmetry. So there exists a vector $c \in \rr^q$ such that $\stab(G[V_1]) =
-\stab(G[V_1]) + 2c$. The origin and all unit vectors $e_i$ for $1 \leq i \leq
q$ are vertices of $\stab(G[V_1])$, thus we must have $c = (\frac{1}{2},
\ldots, \frac{1}{2})$. This means $\stab(G[V_1]) = [0,1]^q$, which in turn
yields $G[V_1] \cong \overline{K_q}$.
\end{proof}

\begin{corollary}\label{hansen_threshold_3d}
If $G$ is a threshold graph, then $\hansen(G)$ satisfies the $3^d$-conjecture
with equality.
\end{corollary}

Kalai suggests in~\cite{kalai1989number} that the Hanner polytopes
should be the \emph{only} polytopes that satisfy the $3^d$-conjecture with equality.
We will see below that other polytopes at least get close.

We also note that not all Hanner polytopes can be represented as Hansen
polytopes of perfect graphs. For example, the product of two octahedra
$O_3 \times O_3$ is a Hanner polytope but not a Hansen polytope.

\section{Hansen Polytopes of Split Graphs}

Now we will analyze the Hansen polytopes of split graphs. 
It is easy to verify and well-known that all threshold graphs are split  
and that all split graphs are perfect.

\begin{definition}[Split graph]
A graph $G$ is called \emph{split graph} if the node set can be partitioned
into a clique $C$ and a stable set $S$.
\end{definition}


The main result of our paper appears in this section as Theorem \ref{main_thm}.
We will prove it with a partitioning technique for the faces of Hansen
polytopes of split graphs. This partitioning will be described first.

\subsection{Partitioning the faces of Hansen polytopes of split graphs.}
\label{subsec_partition}

Let $G = C \cup S$ be a split graph with $C = \{ c_1, \ldots, c_k \}$ and
$S = \{ s_1, \ldots, s_\ell \}$. A stable set of $G$ is either of the form $A$
or $A \cup \{c_i\}$ for $A \subseteq S$. Similarly, a clique of $G$ must be
either of the form $A$ or $A \cup \{s_j\}$ for $A \subseteq C$. Thanks to the
simple composition of stable sets and cliques of split graphs, we can give a
complete description of the vertices and facets of $\hansen(G)$. In the
following we omit set parentheses of singletons in order to enhance readability.

\begin{itemize}[$\bullet$]
\item The vertices of $\hansen(G)$ will be denoted by 
\begin{compactitem}
	\item[(1)] $(\ve, A)$ with $\ve = \pm$ and $A \subseteq S$,
	\item[(2)] $(\ve, A \cup c_i)$ with $\ve = \pm$, $A \subseteq S$ and $A \cup c_i$ stable. 
\end{compactitem}

\item The facets of $\hansen(G)$ will be denoted by 
\begin{compactitem}
	\item[(1)] $[\ve, A]$ with $\ve = \pm$ and $A \subseteq C$, 
	\item[(2)] $[\ve, A \cup s_j]$ with $\ve = \pm$, $A \subseteq C$ and $C \cup s_j$ being a clique. 
\end{compactitem}
\end{itemize}
We will refer to the different kinds of vertices and facets as
\emph{type-(1)-vertices/-facets} and \emph{type-(2)-vertices/-facets} according
to the enumeration above. In the next step we discuss the vertex-facet
incidences.
By Lemma \ref{lemma_hansen_prop} a vertex of $\hansen(G)$ is contained in a 
facet if and only if they have the same sign and their defining subsets of $V(G)$ 
meet or if they have different signs and the defining subsets are disjoint.
\smallskip

\noindent Type-(1)-facets:
\begin{itemize}[$\bullet$]

\item $(\ve, A) \in [\ve', B] \quad\Longleftrightarrow\quad \ve = -\ve'$

\item $(\ve, A \cup c_i) \in [\ve', B] \quad\Longleftrightarrow\quad (c_i \in B \mbox{ and }
\ve = \ve') \mbox{ or } (c_i \not\in B \mbox{ and } \ve = -\ve')$

\end{itemize}

\smallskip

\noindent Type-(2)-facets:
\begin{itemize}[$\bullet$]

\item $(\ve, A) \in [\ve', B \cup s_j] \quad\Longleftrightarrow\quad (s_j \in A \mbox{ and }
\ve = \ve') \mbox{ or } (s_j \not\in A \mbox{ and } \ve = -\ve')$

\item $(\ve, A \cup c_i) \in [\ve', B \cup s_j] \quad\Longleftrightarrow\\ (\ve' = \ve
\mbox{ and } (c_i \in B) \mbox{ or } (s_j \in A)) \mbox{ or } (\ve' = -\ve \mbox{
and } c_i \not\in B \mbox{ and } s_j \not\in A)$
\end{itemize}
\smallskip

\noindent
Observe that the events $c_i \in B$ and $s_j \in A$ are mutually exclusive if
$A \cup c_i$ is stable and $B \cup s_j$ is a clique. The next two lemmas will
be of good use later on.

\begin{lemma}\label{facet_intersection}
Let $G = C \cup S$ be a split graph. Choose $A,B \subseteq C$ and $U
\subseteq S$ such that $A \cup U$ and $B \cup U$ are cliques. Then we have
\begin{enumerate}[\rm(i)]

\item $[\ve, A \cup U] \cap [\ve, B \cup U] = [\ve, (A \cap B) \cup U] \cap [\ve,
A \cup B \cup U]$

\item $[\ve, A \cup U] \cap [-\ve, B \cup U] \subseteq [\ve, A] \cap [-\ve, B]$

\end{enumerate}
\end{lemma}

\begin{proof}
We skip the proof, which easily follows from the vertex-facet-incidences.
\end{proof}

In particular, part (i) shows that every face can be written using at most two
type-(1)-facets of each sign. Indeed, for $A_1, \dots, A_t \subseteq C$, we get
inductively $\bigcap_{i=1}^t [\ve, A_i] = [\ve, \bigcap_{i=1}^t A_i] \cap [\ve,
\bigcup_{i=1}^t A_i]$. The next definition relies on this fact and will be
essential for arguments in the upcoming parts.

\begin{definition}\label{def_partition}
For a split graph $G = C \cup S$ we define the following four classes of faces of
$\hansen(G)$:
\begin{compactitem}[$\bullet$]
\item {\em Primitive} faces $F$, that are not contained in any type-(1)-facet.
\item {\em Positive} faces $[+,A] \cap [+,B] \cap F$, with $A \subseteq B$ and $F$ primitive.
\item {\em Negative} faces $[-,A] \cap [-,B] \cap F$, with $A \subseteq B$ and $F$ primitive.
\item {\em Small} faces $G$, that are contained in type-(1)-facets of both signs.
\end{compactitem}
\end{definition}

\noindent
This definition gives a partition of the faces
of $\hansen(G)$. For the primitive faces we get a nice characterization
with respect to the containment of special vertices.

\begin{lemma}\label{char_primitive}
Let $G = C \cup S$ be a split graph. A face $F$ of $\hansen(G)$ is
primitive if and only if it contains type-(1)-vertices of both signs.
\end{lemma}

\begin{proof}

$(\Rightarrow)$ Assume $F$ is primitive, i.e., we can write it as \[F =
\bigcap_{i \in I} [+,A_i \cup s_i] \cap \bigcap_{j \in J} [-,B_j \cup s_j]\]
for some multisets $I$ and $J$.  If we had $\{s_i : i \in I\} \cap \{s_j : j
\in J\} \neq \emptyset$, then Lemma \ref{facet_intersection} (ii) would yield a
contradiction to primitivity. Thus, these two multisets must be disjoint. We
get the vertex-facet incidences
\begin{compactitem}

\item $(+,A) \in F \ \Longleftrightarrow \ \{ s_i : i \in I \} \subseteq A
\subseteq S \setminus \{ s_j : j \in J \}$,

\item $(-,A) \in F \ \Longleftrightarrow \ \{ s_j : j \in J \} \subseteq A
\subseteq S \setminus \{ s_i : i \in I \}$.

\end{compactitem}
This means we can always find positive and negative type-(1)-vertices of $F$.

$(\Leftarrow)$ If $F$ has a vertex $(\ve,A)$ it cannot be contained in a facet
$[\ve,B]$ for any $B \subseteq C$ according to the rules above. So if $F$
contains type-(1)-vertices of both signs, it cannot be contained in any
type-(1)-facet. This means $F$ is primitive. 
\end{proof}

\subsection{The number of nonempty faces of Hansen polytopes of split graphs.}
We need the following definition to state the main theorem.

\begin{definition}
Let $G = C \cup S$ be a split graph. Then we denote by $p_G(C,S)$ the
number of partitions of the form $(C^+, C^-, C^0, S^+, S^-, S^0)$ with $C = C^+
\cup C^- \cup C^0$ and $S = S^+ \cup S^- \cup S^0$ such that either $C^+ \cup
C^- \neq \emptyset$ or $S^+ \cup S^- \neq \emptyset$, and the following hold:
\begin{compactitem}
\item[(A)] Every element of $C^+ \cup C^-$ has a neighbor in $S^+ \cup S^-$.
\item[(B)] Every element of $S^+ \cup S^-$ has a nonneighbor in $C^+ \cup C^-$.
\end{compactitem}
\end{definition}

\noindent
In the case of Hansen polytopes of split graphs it turns out that $p_G(C,S)$
is exactly the number of faces that we have additionally to $3^d$.
By $\s(P)$ we denote the number of nonempty faces of the polytope $P$.

\begin{theorem}\label{main_thm}
Let $G = C \cup S$ be a split graph on $d-1$ nodes. Then 
\[ 
    \s(\hansen(G)) = 3^d + p_G(C,S). 
\] 
In particular, Hansen polytopes of split graphs satisfy the $3^d$-conjecture.
\end{theorem}

\begin{proof} Let $\Pi$ be the set of all partitions and $\Pi_A, \Pi_B
\subseteq \Pi$ be the subsets for which (A) and (B) hold, respectively. Observe
that if (A) fails for a partition, that there must be a node in $C^+ \cup C^-$
which is not adjacent to any node in $S^+ \cup S^-$. Thus, this partition
fulfills (B). From this we get $\Pi_A^c \subseteq \Pi_B$, where $\Pi_A^c$ is
the complement of $\Pi_A$ in $\Pi$. Analogously, $\Pi_B^c \subseteq \Pi_A$
holds. This yields by some simple counting and inclusion-exclusion \[ 3^{d-1} =
|\Pi| = |\Pi_A| + |\Pi_B| - |\Pi_A \cap \Pi_B|.\] Since $p_G(C,S) = |\Pi_A \cap
\Pi_B| - 1$, we thus need to show that \[ \s(\hansen(G)) = 3^d +
|\Pi_A \cap \Pi_B| - 1 = 2 \cdot 3^{d-1} + |\Pi_A| + |\Pi_B| - 1.\] For this we
are going to use the partitioning of the face lattice of $\hansen(G)$, that was
introduced in Definition \ref{def_partition}. Let $f_p(G)$ be the number of
primitive faces of $\hansen(G)$, $f_+(G)$ be the number of positive, and
$f_-(G)$ be the number of negative ones. Regarding the small faces, one
observes the following: If $F$ is small, then by definition it is contained in
type-(1)-facets of both signs.  Type-(1)-facets correspond to type-(1)-vertices
of the same sign of the polar polytope (via the usual bijection $F \mapsto
F^\ast$ between the face lattice of a polytope and its polar). Lemma
\ref{char_primitive} yields that $F^\ast$ must be a primitive face of
$\hansen(G)^\ast \cong \hansen(\overline{G})$.  Hence, \[ \s(\hansen(G)) = f_p(G) + f_+(G) + f_-(G) + f_p(\overline{G}) - 1. \] All we
need in order to finish this proof is the following lemma.

\begin{lemma}\label{main_lemma}
In the setting above we have
\begin{enumerate}[\rm(i)]
\item $f_+(G) = f_-(G) = 3^{d-1}$
\item $f_p(G) = |\Pi_A|$ and $f_p(\overline{G}) = |\Pi_B|$
\end{enumerate}
\end{lemma}

From this lemma the theorem obviously follows.
\end{proof}

\begin{proof}[Proof of Lemma \ref{main_lemma}]
For this proof we need to refine the notion of a primitive face. Given multisets
$S^+ \cq \{ s_i : i \in I \}$ and $S^- \cq \{ s_j : j \in J \}$, a primitive
face of the form \[ \bigcap_{i \in I} [+,A_i \cup s_i] \cap \bigcap_{j \in J}
[-,B_j \cup s_j] \] will be called \emph{$(S^+,S^-)$-primitive}.

(i) It is clear that a facet $[\ve,A]$ gets mapped to $[-\ve,A]$ by the
bijection $x \mapsto -x$. We therefore have $f_+(G) = f_-(G)$, and showing
$f_+(G) = 3^{d-1}$ will finish this part of the proof. So let us consider a
positive face $P = [+,A'] \cap [+,A] \cap F$, where $A' \subseteq A \subseteq
C$ and \[ F = \bigcap_{i \in I} [+,A_i \cup s_i] \cap \bigcap_{j \in J} [-,B_j
\cup s_j] \] being primitive. As noted in the proof of Lemma
\ref{char_primitive}, the multisets $\{s_i : i \in I\}$ and $\{s_j : j \in J\}$ are
disjoint and $P$ contains a vertex $(-,X)$ if and only if it is contained in $F$,
i.e., if and only if $\{ s_j : j \in J \} \subseteq X \subseteq S \setminus \{ s_i : i
\in I \}$.
Since there are $3^{|S|}$ many ways to choose two
disjoint subsets from $S$, it suffices to show that for fixed $\{ s_i : i \in I
\}$ and $\{ s_j : j \in J \}$, we have $3^{|C|}$ many positive faces of the
above form. To this end let $F$ be a fixed $(\{s_i : i \in I\},\{s_j : j \in
J\})$-primitive face. For such a face the type-(1)-vertices are determined as
just explained, thus it is enough to find out which type-(2)-vertices belong to
$P$. We can describe them precisely as \[ (+ , X \cup z) \in P \Leftrightarrow
z \in A' \mbox{ and } z \not\in \bigcup_{j \in J} B_j \mbox{ and } \forall i
\in I : (\{z,s_i\} \in E(G) \Rightarrow z \in A_i) \] and similarly \[ (-,X
\cup z) \in P \Leftrightarrow z \not\in A \mbox{ and } z \not\in \bigcup_{i \in
I} A_i \mbox{ and } \forall j \in J : (\{z,s_j\} \in E(G) \Rightarrow z \in
B_j). \] These conditions tell us that for each $z\in C$, either there is an $X\subseteq S$
such that $(+, X \cup z) \in P$ or there is an $X$ such that $(-,X \cup z) \in
P$, or none of these is true. Furthermore, all three cases can be controlled
independently and so we get the desired $3^{|C|}$ positive faces for fixed $\{
s_i : i \in I \}$ and $\{ s_j : j \in J \}$.

(ii) Each partition of $G$ that satisfies (A), automatically satisfies (B) for
$\overline{G}$, and the other way around. It is therefore enough to prove
$f_p(G) = |\Pi_A|$. This will be done by constructing a bijection $\mathcal{P}
\rightarrow \Pi_A$, where $\mathcal{P}$ is the set of all primitive faces of
$\hansen(G)$. For this purpose, we partition the domain and range as follows:

\begin{compactitem}

\item Denote by $\mathcal{P}(S^+, S^-)$ the set of all $(S^+,S^-)$-primitive
faces. Then \[ \mathcal{P} = \bigcup \big\{ \mathcal{P}(S^+,S^-) : S^+, S^-
\subseteq S \mbox{ disjoint and } S^+ \cup S^- \neq \emptyset \big\} \] is a
partition of $\mathcal{P}$.

\item Let $\Pi_A(S^+,S^-)$ be the set of all partitions of $G$ that satisfy (A)
and have $S^+, S^-$ fixed (so only $C^+, C^-$ vary). Then \[ \Pi_A = \bigcup
\big\{ \Pi_A(S^+,S^-) : S^+, S^- \subseteq S \mbox{ disjoint and } S^+ \cup S^-
\neq \emptyset \big\} \] is a partition of $\Pi_A$.

\end{compactitem}
From now on let $S^+, S^- \subseteq S$ be disjoint and $S^+ \cup S^- \neq
\emptyset$. We will describe mappings \[ \Psi_{(S^+,S^-)} :
\mathcal{P}(S^+,S^-) \rightarrow \Pi_A(S^+,S^-) \] and \[ \Phi_{(S^+,S^-)} :
\Pi_A(S^+,S^-) \rightarrow \mathcal{P}(S^+,S^-), \] that will turn out to be
inverse to each other. This of course shows that there exists a bijective
correspondence between different parts of the partitions, which allows us to
conclude the existence of a bijection $\mathcal{P} \rightarrow \Pi_A$. Define
$\Psi_{(S^+,S^-)}$ to be \[ \Psi_{(S^+, S^-)} : F \mapsto (C^+, C^-, C^0, S^+,
S^-, S^0), \] and for $\ve = \pm$ let \begin{equation}\label{image_partition}
C^\ve \cq \big\{ c \in C : (\ve,(S^\ve \setminus N(c)) \cup c) \in F \mbox{ and }
\forall\, J \subseteq S : (-\ve, J \cup c) \not\in F \big\}. \end{equation} Here
$N(c)$ again stands for the neighborhood of $c$ in $G$. On the other hand
define $\Phi_{(S^+,S^-)}$ to be \begin{align*} \Phi_{(S^+,S^-)} & : (C^+, C^-,
C^0, S^+, S^-, S^0) \mapsto\\ &\bigcap_{s \in S^+} [+, A_s' \cup s] \cap [+,
A_s \cup s] \cap\!\! \bigcap_{s \in S^-} [-, B_s' \cup s] \cap [-,B_s \cup s],
\end{align*} where $A_s' \cq C^+ \cap N(s)$, $A_s \cq N(s) \setminus C^-$,
$B_s' \cq C^- \cap N(s)$ and $B_s = N(s) \setminus C^+$. Let us use the
 abbreviations $\psi \cq \Psi_{(S^+, S^-)}$ and $\phi \cq \Phi_{(S^+, S^-)}$
for the rest of this proof.

Then we have $\psi \circ \phi = \mbox{id}_{\mathcal{P}(S^+,S^-)}$: Given a partition $\pi = (C^+,
C^-, C^0, S^+, S^-, S^0)$ it is sufficient to prove $\pi \subseteq \psi \circ
\phi (\pi)$, where inclusion is to be understood componentwise. This is
because both $\pi$ and its image are partitions by construction. Let us denote
the first component of the image by $D^+$, the second by $D^-$, and the third by
$D^0$. We begin by explaining why $C^+ \subseteq D^+$. If $c \in C^+$, then by
definition $c \in D^+$ only if
\begin{compactitem}
\item the vertex $v = (+, (S^+ \setminus N(c)) \cup c) \in \phi(\pi)$ and
\item for all $J \subseteq S$ the vertex $w_J = (-, J \cup c) \not\in \phi(\pi)$.
\end{compactitem} 
Concerning the first item, one observes that the stable set $(S^+
\setminus N(c)) \cup c$ does not hit any of the $B_s \cup s$, so $v$ is
 contained in all of the facets with a negative sign. For the facets with
a positive sign the containment is clear if $c \in A_s'$, and in case $c \not\in
A_s'$, we have $c \not\in N(s)$, i.e., $s \in S^+ \setminus N(c)$. Regarding the
second item, by (A) there exists a neighbor $s \in S^+ \cup S^-$ of $c$. If $s
\in S^+$, then $c \in C^+ \cap N(s) = A_s'$ and therefore $c \in A_s' \cup
s$ which rules out that $(-, J \cup c) \in \phi(\pi)$. If $s \in S^-$, then $c
\not\in B_s'$ by construction. So if $w_J \in \phi(\pi)$, we must have $s \in
J$ which contradicts $J \cup c$ being stable. These observations about the
two items above yield $c \in D^+$. The inclusion $C^- \subseteq D^-$ can be
proved similarly.

We continue by explaining $C^0 \subseteq D^0$, so assume $c \in C^0$. If $c
\not\in N(S^+ \cup S^-)$, then $N(c) \cap S \subseteq S^0$ and we get by the
vertex-facet incidences $(+, S^+ \cup c)$, $(-, S^- \cup c) \in \phi(\pi)$, and
in addition $c \in D^0$. If $c \in N(S^+ \cup S^-)$, we can assume w.l.o.g.\
that $\{c,s\} \in E(G)$ for some $s \in S^+$. Then $c \in C^0 \cap N(s)
\subseteq A_s$, which means $(-, J \cup c) \not\in \phi(\pi)$. But we also must
have $c \not\in A_s'$, from which we get $(+, J \cup c) \not\in \phi(\pi)$,
since $s \not\in J$ if $J \cup \{c\}$ is stable. This shows $c \in D^0$, and we
therefore have $\psi \circ \phi = \mbox{id}_{\mathcal{P}(S^+, S^-)}$.

Furthermore, we can deduce $\phi \circ \psi = \mbox{id}_{\Pi_A(S^+,S^-)}$:
Given a primitive face \[ F = \bigcap_{s \in S^+} [+, A_s' \cup s] \cap [+,A_s
\cup s] \cap \!\! \bigcap_{s \in S^-} [-, B_s' \cup s] \cap [-, B_s \cup s],
\] we need to show $\phi \circ \psi (F) = F$. Both $F$ and its image are
$(S^+, S^-)$-primitive faces. Such faces contain type-(1)-vertices $(\ve, J)$
if and only if $S^\ve \subseteq J \subseteq S \setminus S^{-\ve}$; as usual this
follows from the vertex-facet incidences. So $F$ and $\phi \circ
\psi (F)$ contain the same type-(1)-vertices and thus we only need to show that
they also contain the same type-(2)-vertices.

We will begin by showing that if $(\ve, J \cup c) \in F$, then $(\ve, J \cup c)
\in \phi \circ \psi (F)$. To this end we distinguish two cases. 

$1)$ Assume there exists $K \subseteq S$ such that $(-\ve, K \cup c) \in F$.
This means that $c$ cannot be in $A_s$ or $B_s$ for $s \in S^+$ or $s \in S^-$,
respectively. So because of our assumptions, we must have $S^{-\ve} \subseteq
K$ and $S^\ve \subseteq J \subseteq S \setminus S^{-\ve}$. From this we get that $c$
has no neighbor in $S^{-\ve}$. Altogether, this yields $(\ve, J \cup c) \in
[-\ve, (C^{-\ve} \cap N(s)) \cup s]$ for all $s \in S^{-\ve}$, and $(\ve, J
\cup c) \in [\ve, (C^\ve \cap N(s)) \cup s]$ for all $s \in S^\ve$. Hence,
$(\ve, J \cup c) \in \phi \circ \psi (F)$.

$2)$ The other case is $(-\ve, K \cup c) \not\in F$ for all $K \subseteq S$.
If $s \in S^{\ve}$ is not adjacent to $c$, we must have $s \in J$, i.e.,
$S^{\ve} \setminus N(c) \subseteq J$. According to (\ref{image_partition}) we also have
$c \in C^{\ve}$. So for every $s \in S^\ve$, either $s \in J$ or $c \in C^\ve
\cap N(s)$. From this we get that $(\ve, J \cup c)$ is contained in every facet
defining $\phi \circ \psi (F)$ of sign $\ve$. Since $J \cap S^{-\ve} =
\emptyset$, we conclude that $(\ve, J \cup c)$ is also contained in every facet
of sign $-\ve$. This proves $(\ve, J \cup c) \in \phi \circ \psi (F)$.

Finally, we need to prove that if $(\ve, J \cup c) \in \phi \circ \psi (F)$,
then $(\ve, J \cup c) \in F$. Again, we distinguish between two cases. We
know from the vertex-facet incidences that $J \subseteq S \setminus S^{-\ve}$ for all
$(\ve, J \cup c) \in \phi \circ \psi (F)$.

$1)$ Let $S^\ve \subseteq J$. For the sake of contradiction assume $(\ve, J
\cup c) \not\in F$. Then it is easy yet tedious to show that one must have
$(\ve, (S^{\ve} \setminus N(c)) \cup c) \not\in F$. 
(For this recall that $J \cup c$ is stable and that the facets defining $F$ are induced by
cliques, and then prove the contrapositive statement.)
This means that $c \in D^{\ve}$,
where $D^{\ve}$ is again a component of $\psi(F)$. From this in turn we can
conclude that $c \not\in D^{\ve} \cup N(s)$ and $c \not\in N(s) \setminus D^{\ve}$ for some
$s \in S^{-\ve}$, i.e., in particular $c \in N(s)$. Therefore $c \in
(D^0 \cup D^{-\ve}) \cap N(s)$, so $(\ve, J \cup c) \not\in [-\ve, (N(s) \setminus
D^{\ve}) \cup s]$. But this contradicts $(\ve, J \cup c) \in \phi(\psi(F))$.

$2)$ If on the other hand $S^\ve \not\subseteq J$, then there exists $s \in
S^{\ve}$ with $s \not\in J$. Because $(\ve, J \cup c) \in \phi(\psi(F))$ we
then must have $c \in D^{\ve} \cap N(s)$, where $D^{\ve}$ is a component of
$\phi(F)$. So in particular $c \in D^{\ve}$, which of course means $(\ve,
(S^{\ve} \setminus N(c)) \cup c) \in F$. Now it can be easily (but again tediously)
deduced that $(\ve, J \cup c) \in F$.

This shows $\psi \circ \phi = \mbox{id}_{\Pi_{A}(S^+,S^-)}$, and therefore
establishes the bijection and finishes the proof.
\end{proof}

In particular this theorem says that the partition of the split graph does not
play any role in the number of vertices of the corresponding Hansen polytope.
So instead of $p_G(C,S)$ we will write $p_G$ from now on. What we know about
this function is summarized by the following corollary.

\begin{corollary}
Let $G = C \cup S$ be a split graph on $d-1$ nodes. Then
\[
\s(\hansen(G)) = 3^d + 16 \cdot \ell, \quad \mbox{for some } \ell \in \mathbb{N},
\]
with $\ell = 0$ if and only if $G$ is threshold.
\end{corollary}

\begin{proof}
Let us first establish that $p_G = 16 \cdot \ell$. Assume that $C = C^+ \disjoint
C^- \disjoint C^0$ {and} $S = S^+ \disjoint S^- \disjoint S^0$
is given. If $C^+ \cup C^- = \emptyset$, then (B) is only satisfied if $S^+
\cup S^- = \emptyset$. Similarly, if $S^+ \cup S^- = \emptyset$, we have $C^+
\cup C^- = \emptyset$ because of (A). In both cases we deal with the trivial
partition $C^0=C$, $S^0=S$ that is not counted by $p_G$, and thus can be
ignored. If $C^+ \cup C^- = \{ c \}$, then by (A) there exists a neighbor of
$c$ in $S^+ \cup S^-$. By (B) again, this neighbor must have a nonneighbor in
$C^+ \cup C^-$, which clearly cannot be. So also this case is not counted by
$p_G$ and can be ignored as well. By similar reasoning, we can disregard the
case $S^+ \cup S^- = \{ s \}$. Therefore, we must have $|C^+ \cup C^-| \geq 2$
and $|S^+ \cup S^-| \geq 2$. Since we can assign the elements to $C^+, C^-$ or
$S^+, S^-$ in an arbitrary way, we must have $p_G = 16 \cdot \ell$. 

Now $\ell = 0$ if and only if $p_G = 0$. But if $G$ has a path on four nodes
$P_4$ as an induced subgraph, then the partition where $C^+$ is the two middle
nodes of $P_4$, $S^+$ is the two endpoints and $C^-=S^-=\emptyset$, satisfies the
conditions (A) and (B). So if $\ell = 0$, then $G$ is a split graph with no
induced path of four nodes. But by Theorem 1.2.4 in
\cite{mahadev1995threshold}, this happens exactly when $G$ is threshold. On the
other hand, if $G$ is threshold then $\hansen(G)$ is a Hanner polytope by
Theorem \ref{hansen_threshold}, so $\ell = 0$.
\end{proof}

\subsection{High-dimensional Hansen polytopes with few faces.} In the rest of this
section we will study a construction that leads us to high-dimensional Hansen
polytopes with few faces. To this end, consider a threshold graph $T$ on $m$
nodes and a split graph $G = C \cup S$ on $n$ nodes. We construct a new
graph $G \ltimes T$ by taking the union of $G$ and $T$ and adding edges between
every node of $C$ and every node of $T$. Figure \ref{fig_constr} is an illustration of
our construction with $G$ being the path on four nodes.

\begin{figure}[ht]
\begin{center}
\begin{tikzpicture}
\draw (-3,0) -- (3,0);
\filldraw[fill=white] (-3,0) circle (.1cm); 
\filldraw[fill=black] (-1,0) circle (.1cm);
\filldraw[fill=black] (1,0) circle (.1cm);
\filldraw[fill=white] (3,0) circle (.1cm);
\draw (-1,0) -- (-1.3,-1) -- (1,0); 
\draw (-1,0) -- (-0.9,-1.3) -- (1,0); 
\draw (-1,0) -- (0,-1.1) -- (1,0); 
\draw (-1,0) -- (1.2,-1.5) -- (1,0); 
\draw (0,-1.3) ellipse (2.2 and .6);
\coordinate [label=left:$T$] (A) at (0,-1.5);
\end{tikzpicture}
\end{center}
\caption{Appending a threshold graph to a split graph}\label{fig_constr}
\end{figure}
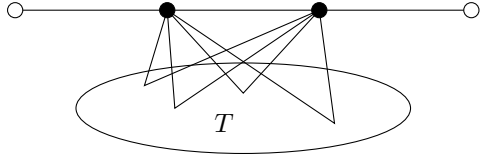

\noindent It is clear that the resulting graph is again a split graph and therefore perfect.

\begin{proposition}
Let $G = C \cup S$ be a split graph on $n$ nodes. Then, for any given
threshold graph $T$ on $m$ nodes, we have
\[
\s(\hansen(G \ltimes T)) = 3^{m+n+1} + p_G.
\]
This means $p_{G \ltimes T} = p_G$, so $p_{G \ltimes T}$ is independent of $T$. 
\end{proposition}

\begin{proof}
By definition the threshold graph $T$ can be built by successive adding of
isolated and dominating nodes. This induces an ordering on the nodes $v_1,
\ldots, v_m$ of $T$. Let $C_T \cq \{ v_i : v_i \mbox{ dominating at step } i
\}$ and $S_T \cq \{ v_i : v_i \mbox{ isolated at step } i \}$. This splits $T$
into a clique $C_T$ and a stable set $S_T$, which in turn splits $G \ltimes T$
into $C \cup C_T$ and $S \cup S_T$. By construction any node in $C_T$ and $S_T$
is connected to all nodes in $C$ and none in $S$. Now consider a partition
$(C^+ \cup C^- \cup C^0, S^+ \cup S^- \cup S^0)$ of $G \ltimes T$ that is
counted by $p_{G \ltimes T}(C \cup C_T, S \cup S_T)$. By (A) for all $x \in
(C^+ \cup C^-) \cap C_T$ there exists a neighbor $y \in (S^+ \cup S^-) \cap
S_T$, which means that in $T$ the once isolated node $y$ was inserted before
the once dominating node $x$. On the other hand, by (B), any given node $y \in
(S^+ \cup S^-) \cap S_T$ has to have a nonneighbor $z \in (C^+ \cup C^-) \cap
C_T$. Such a node $z$ was used before $y$ in the construction of $T$. These two
observations can only hold in the case $(C^+ \cup C^-) \cap C_T = \emptyset =
(S^+ \cup S^-) \cap S_T$. Therefore, for this partition we have $C_T \subseteq
C^0$ and $S_T \subseteq S^0$, which implies that $p_{G \ltimes T}(C \cup C_T, S
\cup S_T) = p_G$.
\end{proof}

This finally yields a series of high-dimensional Hansen polytopes with very few
faces.

\begin{corollary}
Let $P_4$ be a path on four nodes and $T$ be an arbitrary threshold graph on
$m$ nodes. Then 
\[ 
	\s(\hansen(P_4 \ltimes T)) = 3^{m + 5} + 16. 
\]
\end{corollary}

\begin{proof}
Determining $p_{P_4 \ltimes T} = p_{P_4} = 16$ is an easy counting exercise.
\end{proof}

\bigskip

\noindent
\emph{Acknowledgements.} This work has been started while the first two authors
enjoyed the hospitality of the Centre de Recerca Matem\`atica in Barcelona. We
are grateful to Michael Joswig and Eugen Gawrilow for the development of 
\texttt{polymake}~\cite{polymake}: Without this software a considerable part of our work
would not have been possible. 


\end{document}